\documentclass[10pt]{amsart}
\usepackage{amssymb}
\usepackage{dsfont}
\usepackage{amscd}
\usepackage[mathscr]{euscript} 
\usepackage[all]{xy}
\usepackage{url}
\usepackage{stmaryrd}
\usepackage{comment}
\usepackage[retainorgcmds]{IEEEtrantools}
\usepackage{hyperref}
\title{Difference modules and difference cohomology}
\author[M. CHA{\L}UPNIK]{Marcin Cha{\l}upnik$^{\dagger}$}
\thanks{$^{\dagger}$ Supported by the Narodowe Centrum Nauki  grant no. 2015/19/B/ST1/01150}
\address{$^{\dagger}$ Instytut Matematyki\\Uniwersytet Warszawski\\Warszawa\\Poland}
\email{mchal@mimuw.edu.pl}
\author[P. KOWALSKI]{Piotr Kowalski$^{\spadesuit}$}
\thanks{$^{\spadesuit}$ Supported by the Narodowe Centrum Nauki grants no. 2015/19/B/ST1/01150 and  2015/19/B/ST1/01151.}
\address{$^{\spadesuit}$ Instytut Matematyczny\\
Uniwersytet Wroc{\l}awski\\
Wroc{\l}aw\\
Poland}
\email{pkowa@math.uni.wroc.pl} \urladdr{http://www.math.uni.wroc.pl/\textasciitilde pkowa/ }

\DeclareMathOperator{\colim}{colim}
  
\DeclareMathOperator{\gl}{GL}  \DeclareMathOperator{\id}{id}
 \DeclareMathOperator{\fr}{Fr} 
  \DeclareMathOperator{\fix}{Fix}

 \DeclareMathOperator{\alg}{alg}

\DeclareMathOperator{\rat}{rat}

\DeclareMathOperator{\homm}{Hom}\DeclareMathOperator{\modd}{Mod}\DeclareMathOperator{\ext}{Ext}
\DeclareMathOperator{\Mod}{Mod}
\DeclareMathOperator{\Alg}{Alg}\DeclareMathOperator{\Hom}{Hom}

\newtheorem{theorem}{Theorem}[section]
\newtheorem{prop}[theorem]{Proposition}
\newtheorem{lemma}[theorem]{Lemma}
\newtheorem{cor}[theorem]{Corollary}

\theoremstyle{definition}
\newtheorem{definition}[theorem]{Definition}
\newtheorem{example}[theorem]{Example}
\newtheorem{remark}[theorem]{Remark}

\begin{document}
\newcommand{\lili}{\underleftarrow{\lim }}
\newcommand{\coco}{\underrightarrow{\lim }}
\newcommand{\twoc}[3]{ {#1} \choose {{#2}|{#3}}}
\newcommand{\thrc}[4]{ {#1} \choose {{#2}|{#3}|{#4}}}
\newcommand{\Zz}{{\mathds{Z}}}
\newcommand{\Ff}{{\mathds{F}}}
\newcommand{\Cc}{{\mathds{C}}}
\newcommand{\Rr}{{\mathds{R}}}
\newcommand{\Nn}{{\mathds{N}}}
\newcommand{\Qq}{{\mathds{Q}}}
\newcommand{\Kk}{{\mathds{K}}}
\newcommand{\Pp}{{\mathds{P}}}
\newcommand{\ddd}{\mathrm{d}}
\newcommand{\Aa}{\mathds{A}}
\newcommand{\dlog}{\mathrm{ld}}
\newcommand{\ga}{\mathbb{G}_{\rm{a}}}
\newcommand{\gm}{\mathbb{G}_{\rm{m}}}
\newcommand{\gaf}{\widehat{\mathbb{G}}_{\rm{a}}}
\newcommand{\gmf}{\widehat{\mathbb{G}}_{\rm{m}}}
\newcommand{\ka}{{\bf k}}
\newcommand{\ot}{\otimes}
\newcommand{\si}{\mbox{$\sigma$}}
\newcommand{\ks}{\mbox{$({\bf k},\sigma)$}}
\newcommand{\kg}{\mbox{${\bf k}[G]$}}
\newcommand{\ksg}{\mbox{$({\bf k}[G],\sigma)$}}
\newcommand{\ksgs}{\mbox{${\bf k}[G,\sigma_G]$}}
\newcommand{\cks}{\mbox{$\mathrm{Mod}_{({A},\sigma_A)}$}}
\newcommand{\ckg}{\mbox{$\mathrm{Mod}_{{\bf k}[G]}$}}
\newcommand{\cksg}{\mbox{$\mathrm{Mod}_{({A}[G],\sigma_A)}$}}
\newcommand{\cksgs}{\mbox{$\mathrm{Mod}_{({A}[G],\sigma_G)}$}}
\newcommand{\catgrs}{\mbox{$\modd_{{\bf G}}^{\sigma}$}}
\newcommand{\crats}{\mbox{$\mathrm{Mod}^{\rat}_{(\mathbf{G},\sigma_{\mathbf{G}})}$}}
\newcommand{\crat}{\mbox{$\mathrm{Mod}_{\mathbf{G}}$}}
\newcommand{\cratinv}{\mbox{$\mathrm{Mod}^{\rat}_{\mathbb{G}}$}}
\newcommand{\ra}{\longrightarrow}

\begin{abstract}
We give some basics about homological algebra of difference representations. We consider both the difference discrete and the difference rational case. We define the corresponding cohomology theories and show the existence of spectral sequences relating these cohomology theories with the standard ones.
\end{abstract}

\maketitle

\section{Introduction}
In this article, we initiate a systematic study of module categories in the context of difference algebra. Our set-up is as follows. We call an object, such as a ring, a group or an affine group scheme, \emph{difference} when it is additionally equipped with an endomorphism. Hence  a difference ring is just a ring together with a ring endomorphism etc. \emph{Difference algebra} (that is, the theory of difference rings) was initiated by Ritt and developed further by Cohn (see \cite{cohn}). This general theory was motivated by the theory of \emph{difference equations} (they may be considered as a discrete version of differential equations).

We introduce and investigate a suitable category of representations of difference (algebraic) groups which takes into account the extra difference structure.
As far as we know, this quite natural field of research was explored only in \cite{Kam} and \cite{Wib1}.
We discuss the relation between our approach and the one from \cite{Kam} and \cite{Wib1} in Section \ref{secwib}.

We start with discussing the most general case of the category of difference modules over a difference ring in some detail (see Section \ref{secaem}).
However, in the further part of the paper we concentrate on the theory of difference  representations
of a difference group and the parallel (yet more complicated) theory of
difference  representations of difference affine  group schemes.
The emphasis is put on developing the rudiments of homological algebra in these contexts, since our main motivation for studying difference representations is our idea of using difference
language for comparing cohomology of affine group schemes and discrete groups.
Let us now outline our program (further details can be found in Section \ref{ccfs}).

The basic idea is quite general. The  Frobenius morphism extends
to a self--transformation of the identity functor on the category of schemes over
$\Ff_p$. Thus schemes over $\Ff_p$ can be naturally regarded as difference objects. We shall apply this approach to the classical problem of comparing rational and discrete cohomology of affine group schemes defined
over $\Ff_p$. The main result in this area \cite{ratgen} establishes for a reductive algebraic group $\mathbf{G}$ defined over $\Ff_p$ an isomorphism between a certain limit of its rational cohomology groups (called the \emph{stable rational cohomology} of $\mathbf{G}$) and the
discrete cohomology of the group of its $\bar{\Ff}_p$--rational points
(for details, see Section \ref{ccfs}). The main results of our paper (Theorems \ref{disctelthm} and \ref{rasctelthm})
provide an interpretation of stable cohomology in terms of difference cohomology.
Thus, the stable cohomology which was defined ad hoc as a limit is interpreted here
as a genuine right derived functor in the difference framework. We hope to use this
interpretation in a future work which aims to generalize the main theorem of \cite{ratgen} to the case of non--reductive
group schemes. We also hope that this point of view together with Hrushovski's theory of generic Frobenius \cite{HrFro} may lead to an independent and more conceptual proof of the main theorem of \cite{ratgen}. We provide more details of our program in Section \ref{ccfs}.

To summarize, the aim of our article is twofold. Firstly, we develop some basics of
module theory and homological algebra in the difference setting. We believe that some interesting phenomena
already can be observed at this stage. For example, in Remark \ref{lvsr} we point out a striking asymmetry between left and right
difference modules, and in Section \ref{ccfs} we discuss the role of the process of inverting endomorphism.
Thus we hope  that our work will encourage further research in this subject.
Secondly,  we provide a formal framework for applying
difference algebra to homological problems in algebraic geometry in the case of positive characteristic. We hope to use the tools we have worked out in the present paper
in our future work exploring the relation between homological invariants of schematic and discrete objects.

The paper is organized as follows. In Section 2, we collect necessary facts about (non-commutative) difference rings. In Section 3, we deal with the difference discrete cohomology and in Section 4, we consider the difference rational cohomology. In Section 5, we compare our theory with the existing ones and with the theory of spectra from \cite{Chal1}, and we also briefly describe another version of the notion of a difference rational representation (see Definition \ref{defotther}).

We would like to thank the referee for her/his careful reading of our paper and many useful suggestions.

\section{Difference rings and modules}\label{secaem}
In this section, we introduce a suitable module category for difference rings. The theory of difference modules over commutative difference rings has been already considered (see e.g.  \cite[Chapter 3]{levin}),
however our approach is  different than the one from \cite{levin} (we summarize the differences in Remark \ref{lev}).
We recall that  a \emph{difference ring} is a pair $(R,\sigma)$, where $R$ is a ring with a unit (not necessarily commutative), and $\sigma:R\to R$ is a ring homomorphism preserving the unit. A \emph{homomorphism of difference rings} is a ring homomorphism commuting with the distinguished endomorphisms.

Let $(R,\sigma)$ be a difference ring. We call a pair $(M,\sigma_M)$ a \emph{left  difference $(R,\sigma)$--module} if it consists of a left $R$--module $M$ with an additive map $\sigma_M: M\ra M$ satisfying the  condition
\begin{equation}
\sigma_M(\sigma(r)\cdot m)=r\cdot \sigma_M(m),\tag{$\dagger$}
\end{equation}
for any $r\in R$ and $m\in M$ (we explain why did we choose such a condition in Remark \ref{lvsr}). The condition $(\dagger)$ can be concisely
rephrased as saying that the map
\[\sigma_M: M^{(1)}\ra M\]
is a homomorphism of $R$--modules, where $M^{(1)}$ stands for $M$ with the $R$--module structure twisted  by $\sigma$, i.e. $r\cdot m:=\sigma(r)\cdot m$, where $r\in R$ and $m\in M$. The  left difference $(R,\sigma)$--modules form a category  with the morphisms being
the $R$--homomorphisms commuting with the fixed additive endomorphisms satisfying $(\dagger)$.

We have a parallel notion of a \emph{right difference $(R,\sigma)$--module}. This time it
is a right  $R$--module $M$ with an additive map $\sigma_M: M\ra M$
satisfying the condition
\begin{equation}
\sigma_M(m\cdot r)=\sigma_M(m)\cdot \sigma(r)\tag{$\dagger'$},
\end{equation}
which, in terms of the induced $R$--modules, means that the map
\[\sigma_M: M\ra M^{(1)}\]
is $R$--linear.

These categories can be interpreted as genuine module categories, which we explain below. We define the ring of twisted polynomial $R[\sigma]$ as follows.
The underlying Abelian group is the same as in the usual polynomial ring
$R[t]$. However, the multiplication is given by the formula
$$\left(\sum t^ir_i\right)\cdot \left(\sum t^jr'_j\right):=\sum_nt^n\left(\sum_{i+j=n}\sigma^j(r_i)r_j'\right).$$
Then we have the following.
\begin{prop}\label{leftright}
The category of left (resp. right) difference $(R,\sigma)$--modules is equivalent (even isomorphic) to the category of left (resp. right) $R[\sigma]$--modules.
\end{prop}
\begin{proof} Let $M$ be a left difference $R$--module. Then we equip $M$ with
a structure of a left $R[\sigma]$--module by putting
\[\left(\sum t^ir_i\right)\cdot m:=\sum \sigma_M^i(r_i\cdot m).\]
The condition $(\dagger)$ ensures that the commutativity relation in $R[\sigma]$
is respected.  Conversely, for a left $R[\sigma]$--module $N$, we define $\sigma_N$ by
the formula
\[\sigma_N(n):=t\cdot n.\]
Then $\sigma_N:N\to N$ is clearly additive and satisfies $(\dagger)$. The proof for the right modules is similar.
\end{proof}
\begin{remark}\label{lev}
	We summarize here how our definition of a difference module differs from the one in \cite{levin}.
	\begin{enumerate}
		\item Our  base ring of twisted polynomials (defined above) corresponds to the \emph{opposite ring} to the ring of difference operators $\mathcal{D}$  considered in \cite[Chapter 3.1]{levin}. Hence the left difference modules considered in \cite{levin} correspond to our right difference modules.
		
		\item A possible notion of a \emph{right} difference modules (which would correspond to our left difference modules, the choice on which we focus in this paper) is not considered in \cite{levin}.
	\end{enumerate}
\end{remark}
We should warn the reader that the categories of left and right difference modules behave quite differently. For example, since $\sigma: R\ra R^{(1)}$
may be thought of  as a map of $R$--modules,
 $R$ with $\sigma_R:=\sigma$ is a right difference $(R,\sigma)$--module. If $\sigma$
is an automorphism, then obviously $R$ with $\sigma_R:=\sigma^{-1}$ is a left
difference $(R,\sigma)$--module. However, in the general case we do not have any natural
structure of a left difference $(R,\sigma)$--module on $R$.  Since in this paper we are mainly interested in left difference $(R,\sigma)$--modules (a technical explanation is provided in Remark \ref{lvsr}), we would like to construct
a left difference $(R,\sigma)$--module possibly closest to $R$. We achieve this goal by formally inverting the action of $\sigma$ on $R$.
\begin{definition}\label{defrtil}
Let
\[\mathrm{R}_{1-t}: R[\sigma]\ra R[\sigma]\]
be the right multiplication by $(1-t)$. This is clearly a map of left $R[\sigma]$--modules and we define the following left $R[\sigma]$--module:
\[\widetilde{R}:=\mbox{coker}(\mathrm{R}_{1-t}).\]
\end{definition}
Our construction has the following properties.
\begin{prop}\label{rtilda}
Let $\sigma_{\widetilde{R}}$ be the map provided by Proposition \ref{leftright}. Then we have the following.
	\begin{enumerate}
		\item The map $\sigma_{\widetilde{R}}$ is invertible.
		\item If $\sigma$ is an automorphism, then:
$$\left(\widetilde{R},\sigma_{\widetilde{R}}\right)\simeq \left(R,\sigma^{-1}\right).$$
	\end{enumerate}
\end{prop}
\begin{proof}
Since we have the following relation in $\widetilde{R}$:
$$\sum_{i=0}^n t^i r_i=\sum_{i=0}^n t^{i+1}\sigma(r_i),$$
we see that the map $\sum t^i r_i\mapsto \sum t^i \sigma(r_i)$ is the inverse
of $\sigma_{\widetilde{R}}$.

For the second part, we observe first that the map
\[\alpha: (R,\sigma^{-1})\ra \widetilde{R},\]
given by the formula $\alpha(r):=r$, is a homomorphism of left $R[\sigma]$--modules,
since the relation $\sigma^{-1}(r)=tr$ holds in $\widetilde{R}$. Also, the map
\[\beta:\widetilde{R}\ra (R,\sigma^{-1})\]
given by
\[\beta\left(\sum t^i r_i\right):=\sum \sigma^{-i}(r_i)\]
is  a homomorphism of left $R[\sigma]$--modules. We see now that $\alpha$
and $\beta$ are mutually inverse. \end{proof}

From now on,  we focus exclusively on left (difference) modules, hence
we denote by $\modd_{R}^{\sigma}$ the category of left difference $(R,\sigma)$--modules
(or the equivalent  category of left  $R[\sigma]$--modules).
Also, if it causes no confusion we will not refer to endomorphisms in our notation,
i.e. we will usually say ``$M$ is a left difference $R$--module'' (or even ``$M$ is a difference $R$--module'') instead of saying
``$(M,\sigma_M)$ is a left difference $(R,\sigma)$--module''.

We finish this section with an elementary homological computation, which explains (roughly speaking)
the effect of adding a difference structure on homology. We will make this point more precise in the next section.

For a difference   $R$--module $M$, let $M^{\sigma_M}$ (resp. $M_{\sigma_M}$) stand for the Abelian group of invariants (resp. coinvariants) of the action of $\sigma_M$. Explicitly, we have:
\[M^{\sigma_M}=\{ m\in M\ |\ \sigma_M(m)=m\}, \]
and
\[M_{\sigma_M}=M/\langle\sigma_M(m)-m\ |\ m\in M\rangle.\]
Then we have the following.
\begin{prop}\label{gencoh}
For a difference $R$--module $M$, we have:
\[\mbox{\rm{Hom}}_{\modd_{R}^{\sigma}}(\widetilde{R},M)=M^{\sigma_M},\]
$$\ext^1_{{\scriptsize \modd}_{R}^{\sigma}}(\widetilde{R},M)=M_{\sigma_M},$$
$$\ext^{>1}_{\scriptsize{\modd}_{R}^{\sigma}}(\widetilde{R},M)=0.$$
\end{prop}
\begin{proof} Since the map $\mathrm{R}_{1-t}$ is injective,
the complex
\begin{equation*}
 \xymatrix{  0 \ar[r]^{} &  R[\sigma] \ar[r]^{\mathrm{R}_{1-t}}   & R[\sigma] \ar[r]^{}  & 0 }
\end{equation*}
is a free resolution of $\widetilde{R}$.
Then the complex of  Abelian groups
\begin{equation*}
 \xymatrix{  0 \ar[r]^{} &  \homm_{\scriptsize{\modd}_{R}^{\sigma}}(R[\sigma],M) \ar[r]^{(\mathrm{R}_{1-t})^*} & \homm_{\scriptsize{\modd}_{R}^{\sigma}}(R[\sigma],M) \ar[r]^{}  & 0, }
\end{equation*}
which computes our Ext--groups, may be identified with the complex
\begin{equation*}
 \xymatrix{  0 \ar[r]^{} &  M \ar[r]^{\mathrm{L}_{1-t}} & M \ar[r]^{}  & 0,}
\end{equation*}
where $\mathrm{L}_{1-t}$ stands for the left multiplication by the element $(1-t)$.
Thus, the proposition follows. \end{proof}


\section{Difference representations and  cohomology}\label{discat}
Let $(A,\sigma_A)$ be a difference commutative ring and $G$ be a group with an endomorphism $\sigma_G$.
In this section, we apply the results of Section 2 to the  ring $R:=A[G]$,  the group
ring of  $G$ with coefficients in $A$. The ring $R$ with the map
\[\sigma\left(\sum a_i g_i\right):=\sum\sigma_A(a_i)\sigma_G(g_i)\]
is clearly a difference ring. We will often say ``difference representation
of $G$ (over $A$)'' for ``difference $A[G]$--module''.
We observe now that the augmentation map
$\epsilon : A[G]\ra A$ is a homomorphism of difference rings (by this we mean a ring homomorphism commuting with $\sigma$ and $\sigma_A$). Hence, we can endow the left difference $A$--module $\widetilde{A}$ (see Definition \ref{defrtil}) with the ``trivial'' structure of a left difference $A[G]$--module, i.e. we put
\[\left(\sum a_i g_i\right)\cdot a:=\sum a_i\cdot a.\]
\begin{remark}\label{lambdag}
We would like to warn the reader that in contrast to the classical representation theory,
difference representations $(M,\sigma_M)$ correspond  to homomorphisms into the group $\gl_A(M)$
only if $\sigma_M$ is an automorphism. More precisely, if $(M,\sigma_M)$ is a difference $A$--module and $\sigma_M$ is an automorphism, then we have the automorphism  $\widetilde{\sigma_M}$ on $\gl_A(M)$ given by the conjugation:
$$\widetilde{\sigma_M}(\alpha):= \sigma_M^{-1}\circ\alpha\circ \sigma_M.$$
It is easy to see then that  endowing $(M,\sigma_M)$ with the structure of a difference $A[G]$-module is the same as constructing a homomorphism of difference groups
$$\Phi : (G,\sigma_G) \to (\gl_A(M),\widetilde{\sigma_M}).$$
\end{remark}
We are ready now to define the notion of a difference group cohomology.
\begin{definition}
Let $M$ be a  difference $A[G]$--module. We define:
\[H^j_{\sigma}(G,M):=\mbox{Ext}^j_{\scriptsize{\modd}_{R}^{\sigma}}(\widetilde{A},M).\]
\end{definition}
We show below that the 0--th difference cohomology can be described in terms
of invariants.
\begin{prop}\label{disfak}
For any  difference $A[G]$--module $M$, we have:
\[H^0_{\sigma}(G,M)=M^G\cap M^{\sigma_M}.\]
\end{prop}
\begin{proof}
		We observe first  that by the $(\dag$)--condition from Section \ref{secaem}, the $A$--module $M^{G}$ is preserved by $\sigma_M$. Indeed, 		
		for any $m\in M^G$ we have:
		\[g\cdot (\sigma_M(m))=\sigma_M(\sigma_G(g)\cdot m)=\sigma_M(m).\]
		Thus $M^{G}$ is a difference $A$--module and,
		since $G$ acts on $\widetilde{A}$ trivially, we have
	\[\homm_{\scriptsize{\modd}_{A[G]}^{\sigma}}(\widetilde{A},M)=
	\homm_{\scriptsize{\modd}_A^{\sigma}}(\widetilde{A},M^G).\]
By Proposition \ref{gencoh}, we obtain:
	\[\homm_{\scriptsize{\modd}_{A}^{\sigma}}(\widetilde{A},M^G)=(M^G)^{\sigma_M}=
	M^G\cap M^{\sigma_M},\]
	which completes the proof. \end{proof}
		
This description shows possibility of factoring the difference cohomology functor
as the composite of two left exact functors. 	To make this precise, let us consider the chain
of left exact functors:
\[\modd_{A[G]}^{\sigma}\stackrel{K}{\ra} \modd_{A}^{\sigma}
		\stackrel{L}{\ra}\modd_{A},\]
	where
	\[K(M):=\homm_{\scriptsize{\modd}_{A[G]}}(A,M)=M^G\]
	and
	\[L(N):=\homm_{\scriptsize{\modd}_A^{\sigma}}(\widetilde{A},N)=N^{\sigma_N}.\]
	We recall here the fact observed in the proof of Proposition \ref{disfak} that the target category of $K$ is indeed the category
$\modd_{A}^{\sigma}$. Now, Proposition \ref{disfak} can be understood as the following factorization
\[H^0_{\sigma}(G,-)=L\circ K.\]
We would like now to associate  the Grothendieck spectral sequence to the above factorization. To achieve this, we need the following fact.
\begin{lemma}\label{injpr}
	The functor $\epsilon^*:\modd_{A}^{\sigma}{\ra}
	\modd^{\sigma}_{A[G]}$
is left adjoint to $K$. Consequently, the functor $K$ preserves injectives.
\end{lemma}
\begin{proof} The desired adjunction is a natural isomorphism
\[\homm_{\scriptsize{\modd}_{A[G]}^{\sigma}}(\epsilon^*(N),M)\simeq
	\homm_{\scriptsize{\modd}_A^{\sigma}}\left(N,M^G\right),\]
which immediately follows from the fact that $G$ acts trivially on $\epsilon^*(N)$.
Thus $K$ has an exact left adjoint functor, hence it preserves injectives. \end{proof}

The description of the functor $K$ above also shows that for any difference $A[G]$--module $M$, each $H^j(G,M)$ has a natural structure of a difference $A$-module. The endomorphism  of $H^j(G,M)$ can be explicitly described as
 the composite of the following two arrows:
 \begin{equation*}
 	\xymatrix{  H^j(G,M) \ar[r]^{\sigma_G^*\ \ }   & H^j\left(G,M^{(1)}\right) \ar[r]^{(\sigma_M)_*}  & H^j(G,M), }
 \end{equation*}
 where the first one is the restriction map along $\sigma_G$
 (\cite[Chapter 6.8]{WeiHA}), and the second one is the map induced by the $G$--invariant map $\sigma_M:M^{(1)}\to M$.

Then we have the following result, where the invariants and the coinvariants are taken with respect to the difference structure which was just described.
\begin{theorem}\label{spectral}
For any  difference $A[G]$--module $M$ and $j\geqslant  0$, there is a short exact sequence (setting $H^{-1}(G,M):=0$):
\[0\ra H^{j-1}(G,M)_{\sigma}\ra H^j_{\sigma}(G,M)\ra H^j(G,M)^{\sigma}\ra 0.\]
\end{theorem}
\begin{proof}
Since $L,K$ are left exact functors and $K$ takes injective objects to $L$--acyclic ones by Lemma \ref{injpr}, we can construct the  Grothendieck spectral sequence (see e.g. \cite[Chapter 5.8]{WeiHA}) associated to the composite functor $L\circ K$. This spectral sequence converges
to $H^{p+q}_{\sigma}(G,M)$, and its second page has the following form:
\[ E^{pq}_2=\mbox{Ext}^p_{\scriptsize{\modd}_A^{\sigma}}(\widetilde{A}, H^q(G,M)).\]
By Proposition \ref{gencoh}, there are only two nontrivial columns in this page where we have:
\[E_2^{0j}=H^j(G,M)^{\sigma}\ \ \ \mbox{and}\ \ \ E_2^{1j}=H^j(G,M)_{\sigma}.\]
Thus all the differentials vanish and we get the result.
\end{proof}
The above theorem is an efficient tool for computations  of difference cohomology groups. Let us look at some simple examples.
\begin{example}
Let $G=\Zz/p$ be the cyclic group of prime order $p>2$ with an automorphism $\sigma_G $ given
by the formula $\sigma_G(a):=ta$ for some integer $t$ such that $0<t<p$. Let $r$ be the order of $t$ in
the multiplicative group of the field $\Ff_p$ and let further
$A=\ka$ be a field of characteristic $p$.
\begin{enumerate}
\item Let us take $\sigma_A=\id$.  We would like to compute
$$H^*_{\sigma}(\Zz/p,\ka):=\bigoplus_{n=0}^{\infty}H^n_{\sigma}(\Zz/p,\ka)$$
for $(\ka,\id)$ regarded as the trivial difference $\ka[G]$--module. In order to apply Theorem \ref{spectral}, we need to explicitly
describe the endomorphism of $H^*(\Zz/p,\ka)$, let us call it $\sigma_{H^*}$, which comes from the difference structure. When $M$ is a trivial $G$--module,
we have $H^1(G,M)=\mbox{Hom}_{\mathrm{Ab}}(G,M)$ and we obtain:
$$\sigma_{H^1}(\phi)=\sigma_M\circ\phi\circ\sigma_G.$$
Coming back to our example, let us fix a non-zero $y\in H^1(\Zz/p,\ka)$ and let $x\in H^2(\Zz/p,\ka)$ be the image of $y$ by the Bockstein homomorphism.
It is well known (see e.g. \cite[Exercise 6.7.5]{WeiHA}) that we have a ring isomorphism:
\[H^*(\Zz/p,\ka)=S(\ka x)\otimes \Lambda(\ka y),\]
where $S(M)$ is the symmetric power and $\Lambda(M)$ is the exterior power of a $\ka$-module $M$.
Thus we see that $\sigma_{H^1}(y)=ty$ and, by the naturality of the Bockstein homomorphism,
also $\sigma_{H^2}(x)=tx$. Therefore, by the naturality of the multiplicative structure on
group cohomology, for all $j>0$ we obtain:
\begin{equation}
\sigma_{H^{2j}}\left(x^j\right)=t^jx^j,\ \ \ \sigma_{H^{2j-1}}\left(x^{j-1}\ot y\right)=t^j\left(x^{j-1}\otimes y\right).\tag{$\bigstar$}
\end{equation}
Hence we see that $H^{2j}(\Zz/p,\ka)^{\sigma}=\ka x^j$ if and only if $r|j$ (recall that $r$ is the multiplicative order of $t$), and $H^{2j}(\Zz/p,\ka)^{\sigma}=0$ otherwise. A similar conclusion holds for $H^{2j-1}(\Zz/p,\ka)^{\sigma}$, $H^{2j}(\Zz/p,\ka)_{\sigma}$ and
$H^{2j-1}(\Zz/p,\ka)_{\sigma}$. Applying Theorem \ref{spectral}, we get that $H^0_{\sigma}(\Zz/p,\ka)=\ka$ and, for $n>0$, we obtain the following:
\[H^n_{\sigma}(\Zz/p,\ka)=
\left\{
\begin{array}{ll}
\ka\oplus \ka & \ \ \ \ \ \mbox{for $2r|n$};\\
\ka & \ \ \ \ \ \mbox{for $ 2r|n-1$};\\
\ka & \ \ \ \ \ \mbox{for $ 2r|n+1$};\\
0 & \ \ \ \ \ \mbox{otherwise.}
 \end{array}
 \right.\]

\item Let us now elaborate on the above example by adding an automorphism of scalars to the picture. Hence, let $F$ be an automorphism of $\ka$. Then $(\ka, F^{-1})$
is a difference $(\ka,F)[G]$--module and we are interested in its difference cohomology. We recall that $H^1(\Zz/p,\ka)=\mbox{Hom}_{\mathrm{Ab}}(\Zz/p,\ka)$,  which is naturally identified with $\ka$. After choosing $y\in \Ff_p$, we get the same formulas as in $(\bigstar)$ from the item $(1)$ above. Since each $H^n(\Zz/p,\ka)$ is a difference $(\ka,F)$-module, for $c\in \ka$ we obtain the following:
$$\sigma_{H^{2j}}\left(cx^j\right)=F^{-1}(c)t^jx^j,$$
$$\sigma_{H^{2j-1}}\left(cx^{j-1}\ot y\right)=F^{-1}(c)t^j\left(x^{j-1}\otimes y\right).$$
For $a\in\Ff_p\setminus \{0\}$, let $\ka^a$ stand for the eigenspace of $F$ regarded as an $\Ff_p$--linear automorphism of $\ka$ for the eigenvalue
$a$. Dually, let $\ka_a$ be the corresponding ``co--eigenspace'', i.e. the quotient $\Ff_p$--linear space:
$$\ka_a=\ka/\langle F(c)-ca\ |\ c\in \ka\rangle.$$
Therefore, for any non-negative integer $j$, we get by Theorem \ref{spectral}:
\[
H^{2j}_{\sigma}(\Zz/p,\ka)=\ka^{t^j}\oplus \ka_{t^j}, \ \ \
 H^{2j+1}_{\sigma}(\Zz/p,\ka)=\ka^{t^{j+1}}\oplus \ka_{t^{j}}.\]

\item If we consider a special case of the situation considered in the item $(2)$ above, where $A=\ka=\Ff_p^{\alg}$ and $\sigma_A=\fr_{\ka}$ is the Frobenius map, then by the results of \cite[Section 3]{KP4}, the difference module $H^*(\Zz/p,\ka)$ is \emph{$\sigma$-isotrivial}, i.e. we have the following isomorphism of difference modules
    $$H^*(\Zz/p,\ka)\simeq \left(\ka,\fr_{\ka}^{-1}\right)\otimes_{\left(\Ff_p,\id\right)} \left(H^*(\Zz/p,\ka)^{\sigma},\id\right).$$
    (To apply \cite[Fact 3.4(ii)]{KP4}, we need to know that $\sigma_{H^*}$ is a bijection, but it is the case since both $\sigma_G$ and $F$ are automorphisms.) Since $\ka^{\fr}=\Ff_p$, $\ka_{\fr}=0$ and each $H^n(\Zz/p,\ka)^{\sigma}$ is a 1-dimensional vector space over $\Ff_p$, we immediately (i.e. using neither the item $(1)$ nor the item $(2)$ above) get (by Theorem \ref{spectral}) the following isomorphism of $\Ff_p$--linear  spaces:
$$H^*_{\sigma}(\Zz/p,\ka)\simeq S\left(\Ff_p x\right)\otimes \Lambda\left(\Ff_p y\right)=H^*(\Zz/p,\Ff_p),$$
which coincides with the computations made in the item $(2)$.
\end{enumerate}
\end{example}

For a left $A[G]$--module $M$, let us denote by $M^{\infty}$ the induced difference $A[G]$--module, i.e.
$$M^{\infty}:=A[G][\sigma]\ot_{A[G]} M.$$
In order to describe $M^{\infty}$ more explicitly, we slightly extend the notation
introduced in Section 2, by setting $M^{(i)}$ to be the $A[G]$--module $M$ with the structure
twisted by  $\sigma^i$. Then, we have an isomorphism of $A[G]$--modules
\[M^{\infty}\simeq\bigoplus_{i\geqslant  0} M^{(i)}.\]
Under this identification, the difference structure on $M^{\infty}$ is given by the following shift:
\[\sigma_{M^{\infty}}(m_0,\ldots,m_i, 0,\ldots)=(0, m_0,\ldots,m_i,0,\ldots).\]
Let us now investigate the exact sequence from Theorem \ref{spectral} for the difference module $M^{\infty}$.
For this, we introduce the ``stable cohomology groups'' as
\[H^{j}_{\mathrm{st}}(G,M):=\underset{i}{\colim} H^j(G,M^{(i)}),\]
where the maps in the direct system are the restriction maps along $\sigma_G$.
\begin{remark}\label{stvsgen}
We give an interpretation of the stable cohomology in small dimensions.
\begin{enumerate}
\item The $0$-th stable cohomology group
$$H_{\mathrm{st}}^0(G,N)=\bigcup_{n=1}^{\infty}N^{\mathrm{Im}(\sigma^n_G)}$$
may be thought of as  the group of ``weak invariants'' of the action of $G$ on $N$.

\item Suppose that $N$ is a trivial $G$-module. Then we have
$$H^1_{\mathrm{st}}(G,N):=\colim (\Hom(G,N)\to \Hom(G,N)\to \ldots),$$
where the map producing the direct system is induced by $\sigma_G$. Hence $H^1_{\mathrm{st}}(G,N)$ can be considered as the effect of inverting formally the  above endomorphism on $\Hom(G,N)$.
\end{enumerate}
\end{remark}
 These stable cohomology groups play an important role in the comparison between rational and discrete
cohomology in \cite{ratgen}. The fact that, as we will see in a moment, they appear as difference cohomology groups is one of the main motivations for the present work. Namely, when we explicitly describe the action of $\sigma$ on
\[H^*(G,M^{\infty})\simeq H^*(G, \bigoplus_{i\geqslant  0} M^{(i)})\simeq
\bigoplus_{i\geqslant  0} H^*(G,M^{(i)}),\]
 we obtain that (after restricting to the summand $H^*(G,M^{(i)})$) this action is given by the map
\[\sigma_*:H^*(G,M^{(i)})\ra H^*(G,M^{(i+1)})\]
induced by $\sigma$\ on the cohomology. Thus we see that
$H^*(G,M^{\infty})^{\sigma}=0$, and using Theorem \ref{spectral} we get the following.
\begin{theorem}\label{disctelthm}
For any $A[G]$--module $M$ and $j>0$, there is an isomorphism:
	$$H^j_{\sigma}(G,M^{\infty})\simeq H^{j-1}_{\mathrm{st}}(G,M).$$
\end{theorem}
\begin{remark}\label{lvsr}
Apparently, there is no similar description of the stable cohomology in terms of cohomology of right difference modules. The technical obstacle for this is the fact that for a right difference $A[G]$--module $M$, the module of invariants $M^{G} $ is not preserved by $\sigma_M$. Therefore, there is no  Grothendieck spectral sequence analogous to the one which we used in the proof of Theorem \ref{spectral}. This is the main reason for which we have chosen to work with left difference modules in this paper, despite the fact that the condition $(\dag')$ looks more natural than the condition $(\dag)$ (both of them can be found before Proposition \ref{leftright}).
\end{remark}


\section{Difference rational representations and cohomology}\label{secratcat1}

In this section, we introduce difference rational modules and difference rational cohomology. As rational representations and rational cohomology concern representations of algebraic groups, we will consider here representations of \emph{difference algebraic groups}, so we recall this notion first.
Let $\ka$ be our ground field.
\subsection{Difference algebraic groups}\label{secdag}
We take the categorical definition of a difference algebraic group appearing in \cite{Wib1}. When we say ``algebraic group'', we mean ``affine group scheme''. We do not care here about the finite-generation (or finite type) issues: neither in the schematic nor in the difference-schematic meaning. We comment about other possible approaches in Section \ref{secfunc}.

Let $\sigma:\ka\to \ka$ be a field homomorphism. The category of difference $(\ka,\sigma)$--algebras (denoted here by $\Alg_{(\ka,\sigma)}$) consists of commutative \ka--algebras $A$ equipped with ring endomorphisms $\sigma_A$ such that $(\sigma_A)|_{\ka}=\sigma$. A morphism between two $(\ka,\sigma)$-algebras $(A_1,\sigma_1),(A_2,\sigma_2)$ is a $\ka$--algebra morphism $f:A_1\to A_2$ such that
$$\sigma_2\circ f=f\circ \sigma_1.$$

 An \emph{affine difference algebraic group} is defined as a representable functor from the category $\Alg_{(\ka,\sigma)}$ to the category of groups. Note that it is in an exact analogy with the pure algebraic case. Such a functor is represented by a \emph{difference Hopf algebra} which may be defined as $(H,\sigma_H)$, where $H$ is a Hopf algebra over $\ka$, $\sigma^*(H)$ is obtained from $H$ using the base extension $\sigma: \ka \to \ka$ (i.e. $\sigma^*(H)=H\otimes_{\ka}(\ka,\sigma)$) and $\sigma_H:\sigma^*(H)\to H$ is a Hopf algebra morphism (see \cite[Def. 2.2]{Wib1}). Dualizing, we see that a difference algebraic group $\mathcal{G}$ is the same as a pair $(\mathbf{G},\sigma_{\mathbf{G}})$ where $\mathbf{G}$ is an affine group scheme over $\ka$ and $\sigma_{\mathbf{G}}:\mathbf{G}\to \sigma^*(\mathbf{G})$ is a group scheme morphism, where $\sigma^*(\mathbf{G})$ is again obtained from $\mathbf{G}$ using the base extension $\sigma: \ka \to \ka$.

Difference algebraic groups appeared first in the context of model theory (of difference fields) and yielded important applications to number theory (around Manin–-Mumford conjecture) and algebraic dynamics, see e.g. \cite{Hr9}, \cite{algdyn1}, \cite{algdyn2}, \cite{medsc}, \cite{KP4}. Difference algebraic groups also appear as the Galois groups of certain linear differential equations \cite{vhw} and linear difference equations \cite{ow}.

We are mostly interested in the case when $\mathbf{G}$ is defined over the field of constants of $\sigma$ (see Section \ref{ccfs}). In such a case, one can replace the difference field $(\ka,\sigma)$ with the difference field $(\fix(\sigma),\id)$.
Therefore, in the rest of Section \ref{secratcat1}, we  assume that   $\sigma=\id_{\ka}$. In Section \ref{secfunc}, we discuss our attempts to define a more general notion of a difference rational representation, which covers the case of an arbitrary base difference field $(\ka,\sigma)$ (see also Remark \ref{remsrat}).

\subsection{Difference rational representations}\label{sec42}
Let $\mathbf{G}$ be a $\ka$--affine group scheme with an endomorphism $\sigma_{\mathbf{G}}$.  Its representing ring $\ka[\mathbf{G}]$ is a Hopf algebra over $\ka$ with a $\ka$--Hopf algebra endomorphism, denoted here by the same symbol $\sigma_{\mathbf{G}}$.
 We would like to introduce the notion of a difference rational $\bf G$--module.
We recall from classical algebraic geometry (see \cite{jantzen07}) that for a $\ka$--affine group scheme $\mathbf{G}$, a left
rational $\mathbf{G}$--module (or a rational representation of $\mathbf{G}$) is a functor
$$M: \mathrm{Alg}_{\ka}\longrightarrow \mathrm{Mod}_{\ka}$$
such that for any $\ka$--algebra $A$, we have $M(A)=M(\ka)\ot A$, and each $M(A)$ is equipped with a natural
(in $A\in \mathrm{Alg}_{\ka}$) left action of the group $\mathbf{G}(A)$ through
$A$--linear transformations. The left rational $\mathbf{G}$--modules with
the morphisms being the natural transformations form the Abelian category
$\crat$. Given $M\in\crat$, one can construct a natural structure of a \emph{right} $\ka[G]$--comodule on $M(\ka)$. The assignment
$M\mapsto M(\ka)$ gives an equivalence between the category $\crat$
and the category of right $\ka[\mathbf{G}]$-comodules (see \cite[Section I.2.8]{jantzen07}). The inverse is explicitly given by the following construction. An element
$$g\in \mathbf{G}(A)=\Hom_{\mathrm{Alg}_{\ka}}(\ka[\mathbf{G}],A)$$
acts on $M(A)=M(\ka)\ot A$ by the composite
$$(\id\ot m)\circ (\id\ot g \ot \id)\circ(\Delta_M\ot \id),$$
where
$$\Delta_M:M(\ka)\to M(\ka)\ot \ka[\mathbf{G}]$$
is the comodule map on $M(\ka)$, and $m$ is the multiplication
on $A$. From now on, if no confusion can arise, we will identify $M$ with $M(\ka)$.

Let us come back to the situation when $\mathbf{G}$ is additionally equipped
with an endomorphism $\sigma_{\mathbf{G}}$.
A natural adaptation of the concept of a difference representation to the context
of difference algebraic groups is the following.
\begin{definition}
A \emph{difference rational representation} of a difference group $(\mathbf{G},\sigma_{\mathbf{G}})$ is a pair $(M,\sigma_M)$ consisting of a left rational $\mathbf{G}$--module $M$ and a natural transformation $\sigma_M: M\ra M$ such that for each $A\in \mathrm{Alg}_{\ka}$, the $A$--module
$M(A)$ becomes a left difference $A[\mathbf{G}(A)]$--module with
$\sigma_{M(A)}$ being $\sigma_M(A)$, and $\sigma_{A[\mathbf{G}(A)]}$ is given by the following formula:
\[\sigma_{A[\mathbf{G}(A)]}\left(\sum a_i g_i\right):=\sum a_i\sigma_{\mathbf{G}(A)}(g_i).\]

Let $(M,\sigma_M),(N,\sigma_N)$ be rational difference $(\mathbf{G},\sigma_{\mathbf{G}})$-modules. We call a transformation of  functors $f: M\ra N$ a \emph{difference $\mathbf{G}$--homomorphism}, if for any $\ka$--algebra $A$,
$$f(A):M(A)\to N(A)$$
is a homomorphism of difference $A[\mathbf{G}(A)]$--modules.
\end{definition}
Similarly as in Section \ref{discat}, we will often skip the endomorphisms from the notation and simply say that $M$ is a difference rational representation of $\mathbf{G}$. The difference rational representations of $\mathbf{G}$ with difference
$\mathbf{G}$--homomorphisms obviously form a category, which we denote by $\catgrs$.
\begin{remark}\label{comp}
We can find a similar interpretation of our difference rational representations as the one in Remark \ref{lambdag}. We consider $\gl(M)$ as a $\ka$-group functor, see \cite[Section I.2.2]{jantzen07}. In the case when $\sigma_M:M\to M$ is a $\ka$-linear automorphism, it induces the inner automorphism of this  $\ka$-group functor:
    $$\sigma_{\gl(M)}:\gl(M)\to \gl(M).$$
    Then enhancing $(M,\sigma_M)$ with the structure of a $(\mathbf{G},\sigma_{\mathbf{G}})$--module is the same as giving a morphism of difference $\ka$-group functors as below:
    $$\left(\mathbf{G},\sigma_{\mathbf{G}}\right)\to \left(\gl(M),\sigma_{\gl(M)}\right).$$
\end{remark}	
Keeping in mind the results of Section \ref{discat} and
the case of rational representations, we obtain two  equivalent descriptions
of the category $\catgrs$. Analogously as in Section \ref{secaem}, for a rational $\mathbf{G}$--module $M$, we denote by
$M^{(1)}$ the  $\mathbf{G}$--module structure on $M$ twisted by $\sigma_{\mathbf{G}}$. If we take the comodule point of view, then
 the comodule map on $M^{(1)}$ is given by the following composite:
$$(\id\ot \sigma_{\mathbf{G}})\circ\Delta_M:M^{(1)}\to M^{(1)}\otimes \ka[\mathbf{G}].$$
Then we have the following.
\begin{prop}\label{srateq}
Let $\mathbf{G}$ be an affine difference group scheme.  Then 	the following categories are equivalent.
	\begin{enumerate}
\item The category $\catgrs$.
\item The category of pairs $(M,\sigma_M)$, where $M$ is a
rational $\mathbf{G}$--module and $\sigma_M: M^{(1)}\ra M$ is a $\mathbf{G}$--homomorphism.
\item The category of pairs $(M,\sigma_M)$, where $M$ is a right $\ka[\mathbf{G}]$--comodule and $\sigma_M: M\ra M$ is a \ka--linear map satisfying the following identity:
\begin{equation}
\Delta_M\circ \sigma_M=(\sigma_M\ot \sigma_{\mathbf{G}})\circ\Delta_M.\tag{$*$}
\end{equation}
\end{enumerate}
\end{prop}

\begin{remark}\label{remsrat}
A difference rational representation is a natural (in $A\in \Alg_{\ka}$) collection of difference $A[\mathbf{G}(A)]$--modules.
Hence we see that we work in a less general context than the one in Section 3, since
we have no endomorphism on $A$ and neither on $\ka$. It would be tempting to introduce
difference rational representations as  functors on the category of difference
algebras over \ka\ or even over a difference field $(\ka,\sigma)$. The resulting category
is much more complicated, e.g. we have not even succeeded yet in showing
that it is Abelian. Since the simpler approach in this section is sufficient for homological
applications we have in mind, we decided to stick to it in this paper.
We discuss possible generalizations of difference representation theory and its relations with the other approaches
in Section 5.
\end{remark}

\begin{example}\label{3ex}
We point out here three important examples of difference rational $\mathbf{G}$--modules.
\begin{enumerate}
\item The \emph{trivial difference $\mathbf{G}$--module}. Clearly, the $\ka$--algebra unit map
$\ka\to \ka[\mathbf{G}]$ endows $(\ka, \id)$
with the structure of a difference rational $\mathbf{G}$--module.

\item The {\em regular difference $\mathbf{G}$--module} is defined as follows. We put
$$M:=\ka[\mathbf{G}],\ \ \ \sigma_{M}:=\sigma_{\mathbf{G}}.$$
Then the condition $(*)$ in Proposition \ref{srateq}(3) is satisfied, since $\sigma_{\mathbf{G}}$ is a homomorphism of coalgebras.

\item The last example corresponds to the induced module $\ka[G][\sigma]\ot_{\ka[G]} M$
from Section 3. It could be described in terms of cotensor product, but we prefer
the following explicit description. For a rational $\mathbf{G}$--module $M$, we set
$$M^{\infty}:=\bigoplus_{i=0}^{\infty} M^{(i)}$$
as a rational $\mathbf{G}$--module.
Since $(M^{\infty})^{(1)}=\bigoplus_{i=1}^{\infty} M^{(i)}$, the inclusion map
\[\bigoplus_{i=1}^{\infty} M^{(i)}\subset \bigoplus_{i=0}^{\infty} M^{(i)}\]
defines the structure of a difference rational $\mathbf{G}$--module on $M^{\infty}$. Note that this inclusion map is the same as the ``right-shift'' map appearing before Remark \ref{stvsgen}.
\end{enumerate}
\end{example}
 In certain simple cases, the category $\catgrs$ can
 be fully described. The following example should be thought of as the first step
 towards understanding difference rational representations of reductive groups with the Frobenius endomorphism.

Let $\ka$ be a field of positive characteristic $p$, $\gm$ be the multiplicative group over $\ka$ and $\fr:\gm\to \gm$ be the (relative) Frobenius morphism. Then the category $\Mod^{\sigma}_{\gm}$ can be explicitly described.
Let $\Mod^{\Zz,p}_{\ka[x]}$ denote the category of $\Zz$--graded $\ka[x]$-modules satisfying the following condition (for each $j\in \Zz$):
$$xM^j\subseteq M^{pj}.$$
We set  $X:=(\Zz\setminus p\Zz)\cup \{0\}$, and for $j\in X$, we define
$\Mod^{\Zz,p}_{\ka[x],j}$ as the full subcategory of
the category $\Mod^{\Zz,p}_{\ka[x]}$ consisting of modules concentrated in the degrees of the form $p^nj$ for $n\in \Nn$.
Then we have the following.
\begin{prop}\label{rgm}
	The category $\modd_{\gm}^{\sigma}$ admits the following description.
\begin{enumerate}
\item	There is an equivalence of categories
	\[
\modd_{\gm}^{\sigma}\ \simeq\ \Mod^{\Zz,p}_{\ka[x]}.
	\]

\item There is a  decomposition into infinite product
	\[
	\Mod^{\Zz,p}_{\ka[x]}\ \simeq\ \prod_{j\in X} \Mod^{\Zz,p}_{\ka[x],j}.
	\]

\item The category  $\Mod^{\Zz,p}_{\ka[x],0}$ is equivalent to the category of $\ka[x]$--modules, while the category
$\Mod^{\Zz,p}_{\ka[x],j}$ for $j\neq 0$ is equivalent to the category of $\Nn$--graded modules over the graded $\ka$--algebra
$\ka[x]$, where $|x|=1$.
\end{enumerate}
\end{prop}
\begin{proof}  Since $\gm=\mathrm{Diag}(\Zz)$, we can use the results from \cite[Section I.2.11]{jantzen07}. For $M\in \mathrm{Mod}^{\sigma}_{\gm}$, we take a decomposition of the rational module $M\simeq
\bigoplus M_j$ into isotypical rational representations of $\gm$, i.e. each $M_j$ is a direct sum of equivalent irreducible representations such that for each $A\in \Alg_{\ka}$, $a\in\gm(A)$ and $m\in M_j(A)$, we have
$$a\cdot m=a^j m.$$
Then, since $(M_j)^{(1)}=(M^{(1)})_{pj}$,  we have $\sigma_M(M_j)\subseteq M_{pj}$. This turns $M$ into an object of the category
$\Mod^{\Zz,p}_{\ka[x]}$. The rest is straightforward.
\end{proof}

\subsection{Difference rational cohomology}
We would like to develop now some homological algebra in the category $\catgrs$. Firstly, it is obvious that $\catgrs$ is an Abelian
category with the kernels and cokernels inherited from the category $\crat$. However, the existence of enough injectives
is not a priori obvious.
We shall construct injective objects in the category $\catgrs$ by using a particular case of induction. Let $(M,\sigma_M)$ be a $\ka$-linear vector space with an endomorphism. Then, $M\ot\ka[\mathbf{G}]$
with the comodule map $\id\ot \Delta_G$ and the endomorphism $\sigma_M\ot \sigma_{\mathbf{G}}$
satisfies the condition $(*)$ from Proposition \ref{srateq}(3), hence this data defines a difference $\mathbf{G}$--module. This construction is clearly natural, hence it gives rise to a functor
\[\mathrm{\sigma ind}_1^{\mathbf{G}}: \modd_{\ka[x]}\longrightarrow
\catgrs .\]
We will show (similarly to the classical context) that this difference induction functor is right adjoint
to the forgetful functor
\[\mathrm{\sigma res}_1^{\mathbf{G}}:
\catgrs\ra \modd_{\ka[x]}.\]
\begin{prop}\label{preserv}
The functor $\mathrm{\sigma ind}_1^{\mathbf{G}}$ is right adjoint to the functor $\mathrm{\sigma res}_1^{\mathbf{G}}$. Consequently, the functor $\mathrm{\sigma ind}_1^{\mathbf{G}}$ preserves injective objects.
\end{prop}
\begin{proof} We take  $(N,\sigma_N)\in\catgrs $ and $(M,\sigma_M)\in \mathrm{Mod}_{\ka[x]}$.
After forgetting the endomorphisms $\sigma_N, \sigma_M$, we have (by the classical
adjunction) a natural isomorphism
\[\homm_{{\footnotesize \modd}_{\ka}}(N,M)\simeq
\homm_{\footnotesize{\modd}_\mathbf{G}}(N,M\ot \ka[\mathbf{G}]).\]
This isomorphism can be explicitly described as taking a \ka--linear map
$f:N\longrightarrow M$ to the composite $(f\ot \id)\circ\Delta_N$. The inverse
is given by postcomposing with the counit in $\ka[\mathbf{G}]$. Then an explicit
calculation shows that the both assignments preserve morphisms satisfying the condition
$(*)$ from Proposition \ref{srateq}(3), which proves our adjunction.
Preserving injectives is a formal consequence of having exact left adjoint.
\end{proof}
Now we construct injective objects in $\catgrs$ by a standard argument.
\begin{cor}\label{inj}
Any object $M$ in the category $\catgrs$ embeds into an injective object.
\end{cor}
\begin{proof}
Let $\mathrm{\sigma res}_1^{\mathbf{G}}(M)\to I$ be an embedding
in the  category $\modd_{\ka[x]}$, where $I$ is injective. Then we take the chain of embeddings
\[ M\to \mathrm{\sigma ind}_1^{\mathbf{G}}\circ\mathrm{\sigma res}_1^{\mathbf{G}}(M)\to \mathrm{\sigma ind}_1^{\mathbf{G}}(I),\]
and observe that
$\mathrm{\sigma ind}_1^{\mathbf{G}}(I)$ is injective
by Proposition \ref{preserv}.
\end{proof}
Since we have enough injective objects, we can develop now homological algebra in the category $\catgrs$.
\begin{definition}
For a difference rational $\mathbf{G}$--module $M$, we define the \emph{difference rational cohomology groups} (see Example \ref{3ex}(1)) as follows:
\[H^n_{\sigma}(\mathbf{G},M):=\mbox{Ext}^n_{\scriptsize{\modd}_{\mathbf{G}}^{\sigma}}(\ka,M).\]
\end{definition}
We would like to obtain a short exact sequence relating difference rational
and rational cohomology groups. We proceed similarly as in Section \ref{discat}. First, we recall
 that for a rational ${\mathbf{G}}$--module $M$, the $\ka$-vector space
$\Hom_{\scriptsize{\modd}_{\mathbf{G}}^{\sigma}}(\ka,M)$ can be identified with
$$M^{\mathbf{G}}:=\{m\in M\ |\ \Delta_M(m)=m\ot 1\}.$$
By the condition $(*)$ from Proposition \ref{srateq}(3), we immediately get that for a difference rational  $\mathbf{G}$--module $M$, the \ka--module of invariants $M^{{\mathbf{G}}}$ is preserved by $\sigma_M$. Therefore, the functor $(-)^{{\mathbf{G}}}$ can be thought of as a functor
from $\catgrs$ to $\Mod_{\ka[x]}$. Since we can make the following identification:
$$\Hom_{\scriptsize{\modd}_{\mathbf{G}}^{\sigma}}(\ka,M)=M^{{\mathbf{G}}}\cap M^{\sigma_M},$$
we can factor the above Hom-functor through the category $\modd_{\ka[x]}$ as
 $$\Hom_{\scriptsize{\modd}_{\mathbf{G}}^{\sigma}}(\ka,-)=(-)^{\sigma_M}\circ (-)^{{\mathbf{G}}}.$$
 Now, we recall from the proof of Corollary \ref{inj} that for an injective cogenerator
 $I$ of $\modd_{\ka[x]}$, $I\ot \ka[\bf{G}]$ is an injective cogenerator of $\catgrs$. Then we see that
 $$(I\ot \ka[\mathbf{G}])^{{\mathbf{G}}}=I,$$
hence the functor $(-)^{{\mathbf{G}}}$ preserves injectives. Therefore, we can apply the Grothendieck spectral
sequence to our factorization of the functor $\Hom_{\scriptsize{\modd}_{\mathbf{G}}^{\sigma}}(\ka,-)$ and, similarly as in Theorem \ref{spectral}, we get the following.
\begin{theorem}\label{ratseq}
Let $M$ be a difference rational $\mathbf{G}$--module. Then for any $j\geqslant  0$,
there is a short exact sequence (where $H^{-1}(\mathbf{G},M):=0$):
\[0\ra H^{j-1}(\mathbf{G},M)_{\sigma}\ra H^j_{\sigma}(\mathbf{G},M)\ra H^j(\mathbf{G},M)^{\sigma}\ra 0.\]
	\end{theorem}
\begin{proof}
	The proof of Theorem \ref{spectral} carries over to this situation replacing the ring $A[\sigma_A]$ with the ring $\ka[x]$ and the
	discrete  cohomology with the rational cohomology.
	\end{proof}
\begin{example}
We compute rational difference cohomology in the following special case. As a difference rational group, we consider the additive group scheme $\ga$ over $\Ff_p$ ($p>2$) with the Frobenius endomorphism $\fr$, and we take the trivial difference rational $(\ga,\fr)$--module $(\Ff_p,\mbox{id})$.

The ring $H^*(\ga,\Ff_p)$ was computed in \cite[Theorem 4.1]{ratgen} together with a description of the rational action of $\gm$.  In particular, $H^1(\ga,\Ff_p)$ is an infinite dimensional vector space over $\Ff_p$ with a basis $\{a_i\}_{i\geqslant 0}$,
which can be  chosen in such a way that in the action of $\Ff_p[\sigma]$ ($\simeq \Ff_p[x]$) on $$H^1(\ga,\Ff_p)=\Hom(\ga,\ga),$$
 we have $\sigma(a_i)=a_{i+1}$.

 Thus we see that $H^1(\ga,\Ff_p)^{\sigma}=0$
 and $\dim(H^1(\ga,\Ff_p)_{\sigma})=1$. Since $\sigma$ acts trivially
 on $H^0(\ga,\Ff_p)$, we get $\dim(H^0(\ga,\Ff_p)_{\sigma})=1$, and we obtain by Theorem  \ref{ratseq} that
 $$\dim(H^1_{\sigma}(\ga,\Ff_p))=1.$$
 In order to extend our computation, we will use the following description of the graded ring $H^*(\ga,\Ff_p)$ from \cite[Thm 4.1]{ratgen}:
$$H^*(\ga,\Ff_p)\simeq \Lambda(H^1(\ga,\Ff_p))\ot S(\widetilde{H}^1(\ga,\Ff_p)),$$
where $\Lambda$ and $S$ stand respectively for the exterior and symmetric algebra over $\Ff_p$, $\widetilde{H}^1(\ga,\Ff_p)$ is a space with a basis
$\{a_i\}_{i\geqslant 1}$ and its non-zero elements have degree $2$.
Since $\fr$ commutes with algebraic group homomorphisms, the action
 of $\sigma$ on $H^*(\ga,\Ff_p)$ is multiplicative. Hence $\sigma$ acts on
 decomposable elements of  $H^*(\ga,\Ff_p)$ diagonally. Therefore, we have that
$H^j(\ga,\Ff_p)^{\sigma}=0$  for all $j>0$, and we obtain by Theorem  \ref{ratseq} that
\[H^j_{\sigma}(\ga,\Ff_p)\simeq H^{j-1}(\ga,\Ff_p)_{\sigma}\]
for all $j>0$. Taking these facts into account, we can summarize our computations as follows:
\[\dim(H^j_{\sigma}(\ga,\Ff_p))=
\left\{
\begin{array}{ll}
1 & \mbox{for $j=0,1,2$};\\
\infty & \mbox{for $j>2$.}
 \end{array}
 \right.\]
This final outcome may look a bit bizarre, but it coincides with the general philosophy that ``invariants reduce the infinite part of the difference dimension by 1'' (this can be made precise using the notion of an SU-rank, see \cite[Section 2.2]{acfa1}).
 \end{example}
 Continuing the analogy with the discrete situation, we can apply Theorem \ref{ratseq} to the induced difference rational module $M^{\infty}$ (see Example \ref{3ex}(3)). We define, analogously to the discrete case, the ``stable rational cohomology groups'' as:
\[H^{j}_{\mathrm{st}}(\mathbf{G},M):=\underset{i}{\colim} H^j(\mathbf{G},M^{(i)}).\]
Similarly as in Theorem \ref{disctelthm}, we obtain the following.
\begin{theorem}\label{rasctelthm}
	For any rational $\mathbf{G}$--module $M$ and $j>0$, there is an isomorphism:
	$$H^j_{\sigma}(\mathbf{G},M^{\infty})\simeq H^{j-1}_{\mathrm{st}}(\mathbf{G},M).$$
\end{theorem}

\section{Applications, alternative approaches and possible generalisations}\label{secratcat2}

In this section, we discuss  applications of our results to the problem of comparing rational and discrete group cohomology.
We also compare our approach with  the theories of difference representations in \cite{Kam} and \cite{Wib1},  and sketch another (in a way more ambitious) approach to difference representations.

\subsection{Comparison with earlier approach to difference representations}\label{secwib}
Let us compare our construction of difference representations with the existing theories of representations of difference groups in \cite{Wib1} and \cite{Kam}. One sees that Lemma 5.2 from \cite{Wib1} amounts to saying that the category of difference rational representations of $(\mathbf{G},\sigma_{\mathbf{G}})$ considered in \cite{Wib1} is
equivalent to the category of rational representations of $\mathbf{G}$. In fact, in the approach from  \cite{Wib1} and \cite{Kam}, the  difference structure on $\mathbf{G}$ is not encoded in a single representation but rather in some extra structure
on the whole category of representations, namely in the functor
$M\mapsto M^{(1)}$ which twists the $\mathbf{G}$--action  by $\sigma_{\mathbf{G}}$. For example, when the difference group is reconstructed from its representation category through the Tannakian formalism (see \cite{Kam}), this extra structure is used in an essential way.  Hence our approach is, in a sense, more direct. In particular, it allows us to introduce the difference group
cohomology which differs from the cohomology of the underlying
algebraic group. Actually, both of the approaches  build on the same
structure.  Abstractly speaking, we have a category $\mathcal{C}$ with endofunctor
$F$. Then one can consider just the category $\mathcal{C}$ and investigate the effect
of the action of $F$ on it; this is, essentially, the approach initiated in
\cite{Wib1} and \cite{Kam}. On the other hand, one can
introduce, like in our approach, the category $\mathcal{C}^F$, whose objects are the arrows
$$\sigma_M: F(M)\ra M$$
for $M\in \mathcal{C}$. This approach
generalizes the first one, since the construction $M^{\infty}$
(which can be performed in any category with countable coproducts) produces
a faithful functor
$$\mathcal{C}\ra \mathcal{C}^F.$$
On the other hand, our functor $\sigma \mathrm{ind}_{\mathbf{G}}$  produces important objects like injective cogenerators which do not come from $\mathcal{C}$, hence this approach is potentially more flexible and rich.

\subsection{Comparing cohomology, inverting Frobenius and spectra}\label{ccfs}
As we have already mentoned in Introduction, the main motivation for the present work was its possible application to the problem of comparing rational and discrete cohomology. More specifically, let $\mathbf{G}$ be an affine group scheme
defined over $\Ff_p$ and let $M$ be a rational $\mathbf{G}$--module. Then, it is natural to compare the rational cohomology groups $H^j(\mathbf{G},M)$ and
the discrete cohomology groups $H^j(\mathbf{G}(\Ff_{p^n}),M)$. For $\mathbf{G}$
reductive and split over $\Ff_p$, the comparison is given by the celebrated
Cline--van der Kallen--Parshall--Scott theorem  \cite{ratgen} saying that
\[H^j_{\mathrm{st}}(\mathbf{G},M):=\underset{i}{\colim}\ H^j(\mathbf{G},M^{(i)})\simeq \lim_n H^j(\mathbf{G}(\Ff_{p^n}),M),\]
and that the both limits stabilize for any fixed $j\geqslant 0$. Then it was observed \cite[Theorem 4(d)]{arcata} that the right-hand side above (called sometimes \emph{generic cohomology}) coincides with the discrete group cohomology $H^j({\mathbf{G}}(\bar{\Ff}_p),M)$. Our work allows one
to interpret the left-hand side as a right derived functor as well (see Theorem \ref{rasctelthm}). We hope to use  this description in a future work aiming to generalize the main theorem from \cite{ratgen}  to non--reductive algebraic groups.
We expect a theorem on difference cohomology expressing generic cohomology
as a sort of completion of rational cohomology. We hope that the comparison on
difference level should be easier because the limit with respect to the twists is
build--in into the difference theory. Then, one could obtain the theorem on algebraic
groups by taking the $M^{\infty}$--construction (we recall that there is no need for taking ``stable discrete cohomology'' because the Frobenius morphism on a perfect field is an automorphism). This is a subject of our future work.

We would like to point out certain unexpected similarities between Hrushovski's work \cite{HrFro} and the homological results from \cite{ratgen}. In both cases, the situation somehow ``smoothes out'' after taking higher and higher powers of Frobenius. It is visible in the twisted Lang-Weil estimates from \cite[Theorem 1.1]{HrFro} and in the main theorem of \cite{ratgen} above.

At the time being, we can offer another heuristic reasoning supporting our belief that the difference formalism is an adequate tool for the problem of comparing rational and discrete cohomology. Namely, the principal reason why one should not hope for the existence of an isomorphism between rational and generic group cohomology in general is the fact that the Frobenius morphism becomes an automorphism after restricting to the group of rational points over a perfect field. Hence we have
\[H^*(\mathbf{G}(\Ff_q),M)\simeq H^*(\mathbf{G}(\Ff_q),M^{(1)}),\]
while, in general,  there is no reason for the map
\[\sigma_*:H^*(\mathbf{G},M)\ra H^*(\mathbf{G},M^{(1)})\]
to be an isomorphism.  However, the colimit defining  $H^*_{\mathrm{st}}(\mathbf{G},M)$ can be thought of
as the result of making the map $\sigma_*$ invertible (see an example of this phenomenon in Remark \ref{stvsgen}(2)). On the other hand, the process of
inverting the endomorphism $\sigma$ is build in the homological algebra of left
difference modules through the construction of the module $\widetilde{R}$ defined in Section \ref{secaem}. This supports our belief that the category of left difference modules is a relevant tool in this context.

Actually, the first author succeeded in making the connection between the stable cohomology
and the process of inverting Frobenius morphism more precise in an
important special case (see \cite{Chal1}). To explain this idea better, let us come back for a moment to a general categorical context of Section \ref{secwib}.
We assume that we have a category $\mathcal{C}$ with an endofunctor $F$ and a family $\{\mathcal{C}_j\}_{j\in\Zz}$ of full orthogonal
subcategories such that any object in $\mathcal{C}$ is a a direct sum of
objects from $\{\mathcal{C}_j\}_{j\in\Zz}$. Thus we have an equivalence of categories
\[{\mathcal C}\simeq\prod_{j\in  \Zz}{\mathcal C}_j.\]
Moreover, we assume that $F$ takes $\mathcal{C}_j$ into $\mathcal{C}_{pj}$.
This situation is quite common in representation theory
over $\Ff_p$. For example, any central element
of infinite order in ${\bf G}$ produces such a decomposition of the category
of rational representations of ${\bf G}$ with $F$ being the functor of twisting
by the Frobenius morphism (see e.g. Proposition \ref{rgm}). Then we can grade the category
\[{\mathcal C}^*:=\prod_{j\neq 0}{\mathcal C}_j\]
by positive integers, putting
\[{\mathcal C}^*_i:=\prod_{d\in Y}{\mathcal C}_{p^i d}\]
for $i\geqslant 0$, where $Y:=\Zz\setminus p\Zz$.
Let us take now $M=\bigoplus_{i\geqslant 0} M_i$, where $M_i\in\mathcal{C}^*_i$.
Then we see that an object in $(\mathcal{C}^*)^F$ is just a sequence of maps
\[F(M_{i})\ra M_{i+1},\]
hence it produces a
``spectrum of objects of ${\mathcal C}^*$'' \cite{Hov}.
The formalism of spectra is a classical tool which is used to formally invert an endofunctor,
hence it fits well into our context. In \cite{Chal1}, the first author considered ${\mathcal C}$ as the category $\widehat{{\mathcal P}}$ of ``completed'' strict polynomial functors in the sense of \cite{FrSu}, which is closely related to the category of representations
of ${\bf GL}_n$. The category $\widehat{{\mathcal P}}$ has an orthogonal decomposition
\[\widehat{{\mathcal P}}\simeq\prod_{j\geqslant 0}{\mathcal P}_j\]
into the subcategories of strict polynomial functors homogeneous of degree $j$, and $F$ is the ``precomposition with the Frobenius twist functor''.

The first author managed to find \cite[Cor. 4.7]{Chal1} an interpretation of ``stable Ext--groups'' in $\mathcal{P}$ in terms of Ext--groups
in the corresponding category of spectra.  He also  obtained  a version of the main theorem of \cite{ratgen} in ${\mathcal P}$
as an analogue of the Freudenthal theorem \cite[Thm. 5.3(3)]{Chal1}.

Let us now try to compare spectra and difference modules in general. Although the starting categories are very close,
one introduces homological structures in each case in a different way. Namely, in the case of the category of spectra, the formalism of Quillen model categories is used, while in the case of the category of difference modules, we just use its obvious structure of an Abelian category. The important point here is that the resulting Ext--groups are not the same, since in the interpretation of stable cohomology in terms of difference cohomology there is a shift of degree (see Theorem \ref{rasctelthm}). Hence, the relation between these two constructions remains quite mysterious.

\subsection{Functors on the category of difference algebras}\label{secfunc}
We finish our paper with discussing another version of the notion of a difference rational representation. In fact, there is a certain ambiguity  at the very core of difference algebraic geometry. Namely, there are two natural choices for the kind of functors which could be considered as difference schemes:
\begin{enumerate}
\item functors from the category of rings to the category of difference sets;
\item functors from the category of difference rings to the category of sets.
\end{enumerate}
In the case of representable functors (i.e. affine difference schemes) both of the choices above are equivalent by the Yoneda lemma. Thanks to this, a difference group scheme can be unambiguously defined as (the dual of) a difference Hopf algebra. Unfortunately, this ``several choices'' problem re-appears when
one tries to introduce the appropriate notion of a difference representation.
In fact, we made in Section \ref{secratcat1} the ``first choice'' which is simpler and sufficient for the main objectives of our article.  The drawback of this approach is that the difference structure on the module $M(A)$ from Section \ref{sec42} does not depend on
 a possible difference structure on $A$. In other words: there is no natural way of
 turning the functor $M$ into a functor on the category of difference $\ka$--algebras.
 For this reason, the framework of Section \ref{secratcat1} is less general than the one in Section \ref{discat}.
 Thus, it would be tempting to  introduce the notion of a difference rational representation corresponding to the ``second choice'' above.

We will outline now an alternative approach, which is potentially richer but is also much more involved technically.  We fix a difference
field $(\ka,\sigma)$ and consider the category $\Alg_{(\ka,\sigma)}$ of
difference commutative algebras over $\ka$ as in Section \ref{secdag}.
Then, undoubtedly, we want  our  difference representation to be some sort of a functor
\[M: \Alg_{(\ka,\sigma)}\ra \Mod_{\ka[\sigma]},\]
such that $M(A)$ is naturally  a difference $(A,\sigma_A)$--module. Now we need an analogue of the fact that an ordinary
rational representation sends a \ka--algebra $A$ to $A\ot M(\ka)$.  A reasonable choice here seems to be the following:
\[M(A)=A[\sigma]\ot_{\ka[\sigma]} M(\ka),
\]
 since in that case the structure
of an $A[\sigma]$--module on $M(A)$ depends both on $(A,\sigma_A)$
and on $(M,\sigma_M)$.

When we add to this framework a group action, we obtain the following definition.
\begin{definition}\label{defotther}
Let $(\mathbf{G},\sigma_{\mathbf{G}})$ be a difference algebraic group.
We call a  functor
\[M: \Alg_{(\ka,\sigma)}\ra \Mod_{\ka[\sigma]},\]
such that \[M(A)=A[\sigma]\ot_{\ka[\sigma]} M(\ka)\]
a \emph{$\mathbf{G}$--difference representation} (or a \emph{$\mathbf{G}$--difference module}),
 if there is a natural (in $A\in \Alg_{(\ka,\sigma)}$) structure of a difference $A[\mathbf{G}(A)]$--module on $M(A)$.
\end{definition}
 With the above definition, we achieve the level of generality we had in the
discrete case of Section \ref{discat}. However, in order to make the category of such difference representations usable, one
would like to obtain its algebraic description in terms of comodules over coalgebras etc. Unfortunately, the formulae we have obtained so far are quite complicated and do not fit easily into known patterns. For example, it is not
clear how to use them even to show  that  the category under consideration
has enough injective objects. For this reason, in this paper, we decided to adopt the approach corresponding to the ``first choice''.

\bibliographystyle{plain}
\bibliography{harvard}

\end{document}